\documentclass[12pt,twoside,reqno]{amsart}
\usepackage[colorlinks=true,citecolor=blue]{hyperref}
\usepackage{mathptmx, amsmath, amssymb, amsfonts, amsthm, mathptmx, enumerate, color}
\usepackage{graphicx}
\usepackage{epstopdf}
\usepackage{tikz}
\usepackage{enumerate}
\usepackage{enumitem}
\newcommand*\mycirc[1]{%
\begin{tikzpicture}[baseline=(C.base)]
\node[draw,circle,inner sep=1pt,minimum size=3ex](C) {#1};
\end{tikzpicture}}

\newtheorem{theorem}{Theorem}[section]

\newtheorem{proposition}{Proposition}[section]
\theoremstyle{definition}
\newtheorem{definition}{Definition}[section]
\newtheorem{example}{Example}[section]
\newtheorem{remark}{Remark}[section]

\numberwithin{equation}{section}

\begin{document}
\setcounter{page}{1}

\vspace*{1.0cm}
\title[Construction of Spaces with an Indefinite two-Metric]
{Construction of Spaces with an Indefinite two-Metric and Applications}
\author[F. Osmin, F. Kandy, J. Cure]{ Osmin Ferrer Villar$^{1,*}$, Kandy Ferrer Sotelo$^2$, Jaffeth Cure Arenas$^1$}
\maketitle
\vspace*{-0.6cm}

\begin{center}
{\footnotesize {\it

$^1$Department of Mathematics, University of Sucre.
Red door, Sincelejo, Sucre, Colombia.\\
$^2$Center for Basic Sciences, School of Engineering and Architecture, Pontifical Bolivarian University, Monteria, Colombia.

}}\end{center}

\vskip 4mm {\small\noindent {\bf Abstract.}

In this work, we introduce the notion of a two-Krein space and show that, starting from any classical Krein space, it is possible to construct spaces endowed with an indefinite two-inner product (admitting both positive and negative values) (Proposition \ref{k2k}). We develop the theory of two-Krein spaces (Proposition \ref{two-krein}), extending the classical structure and providing new tools for analysis in spaces with an indefinite two-inner product.

It is established that the fundamental decomposition of a Krein space transfers orthogonality to the space with an indefinite two-metric (Proposition \ref{presesrvaort}). Moreover, the properties of the fundamental symmetry of the classical Krein space are carried over to the standardized two-Krein space (Proposition \ref{propiedades j}). It is shown that the sets of positive and negative vectors generate a space with a semi-definite positive and complete two-inner product (Proposition \ref{completez}). One of the most important results in the theory of spaces with an indefinite metric is the equivalence of norms (Theorem \ref{equiv}); in this work, we extend this result to standardized spaces with an indefinite two-metric (Theorem \ref{2equiv}).

Additionally, the notion of function strongly of bounded variation in two-Krein spaces is introduced (Definition \ref{tvar}) and some of their properties are established (Theorem \ref{desigualdades}). It is also shown that the classical definition of bounded variation in two-Hilbert spaces \cite{FCF} is a particular case of the one presented in this work (Remark \ref{2h2k}). Furthermore, we present a technique to construct functions of bounded $t$-variation in standardized two-Krein spaces from functions of bounded variation in Krein spaces (Proposition \ref{va implica 2kva}), and we guarantee that when the $t$-variation of a function is zero, the two-norm evaluated at the images of the function remains constant with respect to $t$. Finally, we show that the class of strongly bounded $t$-variation functions in a standardized two-Krein space can be endowed with the structure of a two-norm (Theorem \ref{dota norma}).

 }

\renewcommand{\thefootnote}{}
\footnotetext{ $^*$Corresponding author.
\par
E-mail addresses: osmin.ferrer@unisucre.edu.co (F. Osmin), kandy.ferrer@upb.edu.co (F. Kandy), jaffeth.cure@unisucrevirtual.edu.co (C. Jaffeth).
}

\section{Introduction}

The notion of the variation of a function over an interval $[a,b]$ originated in the work of the French mathematician Camille Jordan \cite{Jordan}. Since its introduction, the concept of bounded variation has been the subject of multiple generalizations in various directions. Among these, notable extensions include those to functions valued in vector spaces and functions taking values in normed spaces \cite{Mendoza, Chistyakov Metric}, Moreover, the concept has been extended to more general structures, including two-normed spaces \cite{FCF}. These generalizations have significantly broadened the scope of application of the concept.

As spaces with more intricate structures have emerged, particularly those by an indefinite inner product \cite{Azizov, Bognar}, various authors have sought to adapt and extend classical concepts of functional analysis to this new context. A notable example of this line of research is the work presented in \cite{FNG}, which explores this concept within the framework of Krein spaces. This generalization has enabled the transfer of fundamental tools from the classical theory of bounded variation to spaces with an indefinite metric, thereby enriching the study of functions and operators in these settings.

A fundamental problem we address in this work is the possibility of constructing a space endowed with an indefinite two-inner product from a space with an indefinite inner product, thus extending the classical structure of Krein spaces. We show that it is indeed possible to generate a two-Krein space from a classical Krein space, preserving and adapting essential properties such as orthogonal decomposition and fundamental symmetry.

Furthermore, we establish that this construction is not only compatible with the structure of the original space, but also allows for the transfer of relevant analytical properties. In particular, we demonstrate that every function of bounded variation in the Krein space induces a function of bounded $t$-variation function in the standardized two-Krein space. This result shows that the notion of bounded variation in spaces with an indefinite inner product can be viewed as a particular case within the more general theory we develop for spaces with an indefinite two-inner product, thereby consolidating the validity and scope of the proposed generalization.

\section{Preliminaries}
\subsection{Krein spaces}

\begin{definition}\cite{Azizov, Bognar}
A \textbf{Krein space} is a pair $(\mathcal{F}, [\cdot,\cdot])$, where $\mathcal{F}$ is a vector space and $[\cdot,\cdot]$ is an indefinite inner product, such that there exists a direct sum decomposition $\mathcal{F} = \mathcal{F}^{+} \dot{[+]} \mathcal{F}^{-}$  
with $(\mathcal{F}^{+}, [\cdot,\cdot])$ and $(\mathcal{F}^{-}, -[\cdot,\cdot])$ forming Hilbert spaces.

\end{definition}

\begin{definition}  
\cite{Azizov, Bognar}  
Given a Krein space $\mathcal{F}$ with decomposition $\mathcal{F} = \mathcal{F}^{-} \dot{[+]} \mathcal{F}^{+}$, the \textbf{fundamental symmetry} is the operator $\mathcal{J}:\mathcal{F} \to \mathcal{F}$ given by  
\[
\mathcal{J}x = x^{+} - x^{-}.
\]  
\end{definition}

\begin{remark}
The fundamental symmetry has the following properties:  
\begin{itemize}  
\item $\mathcal{J}=\mathcal{J}^{-1}$.  
\item $[\mathcal{J}x,y]=[x,\mathcal{J}y]$ para todo $x,y\in\mathcal{F}$.  
\item $\mathcal{J}$ is an isometric operator.  
\end{itemize}  
\end{remark}

\begin{definition}\cite{Azizov, Bognar}
Let $(\mathcal{F} = \mathcal{F}^{+} \dot{[+]} \mathcal{F}^{-}, [\cdot, \cdot],\mathcal{J})$ be a Krein space. The map \\$[\cdot, \cdot]_{\mathcal{J}}: \mathcal{F} \times \mathcal{F} \to \mathbb{C}$ is given by  
\[
[x, y]_{\mathcal{J}} = [\mathcal{J}x, y], \quad x, y \in \mathcal{F}.
\]  
This map is referred to as the \textbf{$\mathcal{J}$-inner product}.
\end{definition}

\begin{remark}
In any $\mathcal{J}$-inner product, it holds that $[x,x] \geq 0$ and $[x,x] = 0 \iff x = 0$.
\end{remark}

\begin{theorem}\cite{Azizov, Bognar}  
Let $(\mathcal{F} = \mathcal{F}^{+} \dot{[+]} \mathcal{F}^{-}, [\cdot , \cdot ])$ be a Krein space, and consider \(\mathcal{J}\) as the fundamental symmetry linked to the specified decomposition. Then, the following holds: 
$$
\vert [x, y] \vert \leq \Vert x \Vert_{\mathcal{J}} \Vert y \Vert_{\mathcal{J}}, \, \, \, x, y \in \mathcal{F}.
$$
\end{theorem}

\begin{proposition}\cite{Azizov, Bognar}
In the Krein space $(\mathcal{F}=\mathcal{F}^+\dot{[+]}\mathcal{F}^-,[\cdot,\cdot])$, the fundamental symmetry $\mathcal{J}$ determines a norm on $\mathcal{F}$, given by  
$$
\|x\|_{\mathcal{J}} = \sqrt{[x,x]_{\mathcal{J}}}, \quad \forall x \in \mathcal{F}.
$$
This norm is referred to as the \textbf{$\mathcal{J}$-norm} of the Krein space $\mathcal{F}$.
\end{proposition}
\begin{remark}
The Hilbert spaces $(\mathcal{F}^{+},[\cdot,\cdot])$ and $(\mathcal{F}^{-},-[\cdot,\cdot])$ have the following associated norms, respectively:  
$$  
\|x^{+}\|_{+}=\sqrt{[x^{+},x^{+}]},\quad \|x^{-}\|_{-}=\sqrt{-[x^{-},x^{-}]},\quad \text{for all}\ x^{+}\in \mathcal{F}^{+},\ x^{-}\in \mathcal{F}^{-}.  
$$  
\end{remark}

\begin{theorem}\label{equiv}\cite{Azizov, Bognar}
In a Krein space $(\mathcal{F}, [\cdot , \cdot ])$, any two fundamental decompositions induce equivalent norms via their respective fundamental symmetries.
\end{theorem}

\begin{remark}
Let $(\mathcal{F} = \mathcal{F}^{+} \dot{[+]} \mathcal{F}^{-}, [\cdot, \cdot], \mathcal{J})$ be a Krein space, $[a,b]$ an interval, and let $f: [a,b] \to \mathcal{F} = \mathcal{F}^{+} \dot{[+]} \mathcal{F}^{-}$. Taking into account that, for any $t$ in $[a,b]$, $f(t)$ belongs to \\$\mathcal{F} = \mathcal{F}^{+} \dot{[+]} \mathcal{F}^{-}$, from now on we will write the image of $t$ under $f$ as $f(t) = f^{+}(t) + f^{-}(t)$.
\end{remark}

\begin{definition}\cite{FNG}\label{fuertemente de variacion acotada}
Let $(\mathcal{F} = \mathcal{F}^{+} \dot{[+]} \mathcal{F}^{-},[\cdot,\cdot])$ be a Krein space, and let $f$ be a function defined on the interval $[a,b]$. We say that $f$ is \emph{strongly of bounded variation} on $[a,b]$ in $\mathcal{F}$ if
$$
V_a^b(f,(\mathcal{F},[\cdot,\cdot])) = \sup\left\{ \sum_{i=1}^n \left( \|f^{+}(t_i) - f^{+}(t_{i-1})\|_{+} + \|f^{-}(t_i) - f^{-}(t_{i-1})\|_{-} \right) : P \in \mathcal{P}[a,b] \right\}
$$
is finite.
\end{definition}

\subsection{On two-normed spaces}

The study of two-normed spaces has garnered significant interest within the mathematical community, as evidenced by the various generalizations developed from this concept. In our case, we build upon the work carried out by Lewandowska, considering a particular case in which the two-norm is a symmetric mapping taking elements from the same vector space, specifically where $\mathcal{D} = F \times F$. Below, we present the formal definition that will be used in this context.

\begin{definition}\cite{Gahler, lewandowska}
Let $F$ to be a complex vector space of dimension $d$, where  $2\leq d\leq \infty$. A \textbf{two-norm} on $F$ is a function $\mathcal{N}:F \times F\to\mathbb{R}$ that satisfies the following conditions:
\begin{enumerate}
\item[$(2N_1)$] $\mathcal{N}(g,\alpha u)=|\alpha|\mathcal{N}(g,u)$, for all  $\alpha \in \mathbb{C}$;
\item[$(2N_2)$]\label{2N4} $\mathcal{N}(g,u+h) \leq \mathcal{N}(g,u)+\mathcal{N}(g,h)$;
\item[$(2N_3)$] $\mathcal{N}(g,u)=\mathcal{N}(u,g)$.
\end{enumerate}
If $\mathcal{N}$ is a two-norm for $F$, then the pair $(F,\mathcal{N})$ is called a \textbf{two-normed space}. 
\end{definition}

\begin{definition}\cite{cho}
Let $F$ as a complex vector space of dimension $d \geq 2$. A two-inner product is a function
$$
\psi: F \times F \times F \to \mathbb{C}
$$
that satisfies the following conditions:
\begin{enumerate}  \item[$(2\text{I}_1)$] $\psi\left(g_{1}+g_{2}, u , h\right)=\psi\left(g_{1}, u , h\right)+\psi\left(g_{2}, u , h\right)$.
\item[$(2\text{I}_2)$]  $\psi(g, g , h)=\psi(h, h , g)$;
  \item[$(2\text{I}_3)$]  $\psi(u, g , h)=\overline{\psi(g, u , h)}$;
  \item[$(2\text{I}_4)$]  $\psi(\alpha g, u , h)=\alpha \psi(g, u , h)$ for all $\alpha \in \mathbb{C}$;
  \item[$(2\text{I}_5)$]  $\psi(g, g , h) \geq 0$
\end{enumerate}
The pair $(F,\psi)$ is called a two-inner product space (or pre-two-Hilbert space).
\end{definition}

\begin{remark}
Note that in $(2\text{I}_{3})$, $\psi(u, g , h)=\overline{\psi(g, u , h)}$. If $u=g$, we obtain $$\psi(g, g , h)=\overline{\psi(g, g , h)},$$ which guarantees $\psi(g, g , h) \in \mathbb{R}$.
\end{remark}
 \begin{proposition}\label{2schw}\cite{AJO}
 Let $(F,\psi)$ a space with two-inner product and $g, u, h \in F.$ Then it holds that:
$$
|\psi(g, u , h)|^{2} \leq \psi(g, g , h)\psi(u, u , h).
$$
The inequality above is the analogue of the Cauchy-Schwarz inequality in spaces with inner product.
\end{proposition}

\begin{remark}\label{two-norma inducida} Given a space with two-inner product $\left(F,\psi\right)$, we can define a two-norm for $F$, called the {\bf induced two-norm} of the two-inner product given by
$$\mathcal{N}(g,h):=\sqrt{\psi(g,g,h)}\quad g,h\in F.$$
\end{remark}

\begin{definition}\cite{white}
A sequence $\left\{g_{n}\right\}_{n  \in\mathbb{N}}$ in a two-normed space $F$ is called a convergent sequence if there exists a $g \in F$ such that the $\displaystyle\lim_{n\to\infty} \mathcal{N}\left(g_{n}-g, u\right)=0$ for all $u \in F$. If $\left\{g_{n}\right\}_{n  \in\mathbb{N}}$ converges to $x$, we write $g_{n} \to g$ and we call $g$ the limit of $\left\{g_{n}\right\}_{n  \in\mathbb{N}}$.
\end{definition}
\begin{definition}\cite{kazemi}
Let $(F,\mathcal{N}(\cdot,\cdot))$ a two-normed space, $h\in F$. A sequence $\left\{g_{n}\right\}_{n  \in\mathbb{N}}$ in $F$ is called $h$-Cauchy if for all $\epsilon>0$, there exists $N\in\mathbb{N}$, such that if $m,n>N$ then \\$\mathcal{N}\left(g_{n}-g_{m}, h\right)<\epsilon$.
\end{definition}

\begin{definition}\cite{lewandowska}
A space with two-inner product $\left(\mathcal{H},\psi\right)$ is said to be a {\bf two-Hilbert} space if it is complete with respect to the two-norm induced by the two-inner product.
\end{definition}

\begin{definition}
\cite{orto} Let $(F, \psi)$ be a space with a two-inner product, and let $g, u \in F$. We say that $g$ and $u$ are \emph{two-orthogonal} in $F$, denoted by $g \perp u$, if for all $h \in F$ the following holds $\psi(g, u , h) = 0$.
\end{definition}

\begin{definition}\cite{FCF} \label{2kvar}Let $(F,\mathcal{N})$ be a two-normed space, $h \in F$, $[a,b]$ a closed interval, and $f\colon [a,b] \to F$ a function. The $(2,h)$-variation of $f$ over $[a,b]$, denoted by $V_{a}^{b}(f,F,h)$, is defined as:
\[
V_{a}^{b}(f,F,h) = \sup \left\{ \sum_{i=1}^{n} \mathcal{N}( f(t_i) - f(t_{i-1}), h) : P = \{t_0, t_1, \dots, t_n\} \in \mathcal{P}([a,b]) \right\}.
\]
If $V_{a}^{b}(f,F,h) < \infty$, the function $f$ is called of bounded $(2,h)$-variation.
\end{definition}

We now present the notion of an indefinite two-metric, thereby extending the definition given by Gähler \cite{Gahler}.

\section{Main Results
}
We now introduce an indefinite two-metric, derived from the structure of a Krein space, which preserves the interaction between its positive and negative components.
\begin{proposition}\label{k2k}
If $(\mathcal{F} = \mathcal{F}^{+}[\overset{\cdot}{+}]\mathcal{F}^{-}, [\cdot, \cdot], \mathcal{J})$ is a Krein space, then the application \\$\psi \colon \mathcal{F} \times \mathcal{F} \times \mathcal{F} \longrightarrow \mathbb{C}$ defined by
\begin{equation}\label{two-krein ind}
\psi(x,y , z) = [x^{+}, y^{+}] \left\| z^{+} \right\|^{2}_{+} + [x^{-}, y^{-}] \left\| z^{-} \right\|^{2}_{-}
\end{equation}
satisfies the following properties:

\begin{enumerate}[itemsep=0pt,label=\protect\mycirc{\arabic*}]
  \item $\psi\left(x_{1}+x_{2}, y , z\right)=\psi\left(x_{1}, y , z\right)+\psi\left(x_{2}, y , z\right)$.
  \item $\psi(x, x , z)=\psi(z, z , x)$;
  \item $\psi(y, x , z)=\overline{\psi(x, y , z)}$;
  \item $\psi(\alpha x, y , z)=\alpha\psi(x, y , z)$ for all $\alpha \in \mathbb{C}$.
  \item There exist $x_{1}, x_{2} \in X$ such that $\psi(x_{1}, x_{1} , z) > 0$ and $\psi(x_{2}, x_{2} , z) < 0$ for all $z \in \mathcal{F}$.
\end{enumerate}

\end{proposition}
\begin{proof}
Let $x, x', y, z \in \mathcal{F}$ and $\alpha \in \mathbb{C}$. Then,

$\protect\mycirc{1}$
\begin{align*}
\psi(x+x',y,z)&=[x^{+}+x'^{+},y^{+}]\left\| z^{+}\right\|^{2}_{+}+[x^{-}+x'^{-},y^{-}]\left\| z^{-}\right\|^{2}_{-}\\
&=[x^{+},y^{+}]\left\| z^{+}\right\|^{2}_{+}+[x'^{-},y^{-}]\left\| z^{-}\right\|^{2}_{-}+[x'^{+},y^{+}]\left\| z^{+}\right\|^{2}_{+}[x^{-},y^{-}]\left\| z^{-}\right\|^{2}_{-}\\
&=\psi(x,y,z)+\psi(x',y,z).
\end{align*}

$\protect\mycirc{2}$
\begin{align*}
\psi(\alpha x,y,z)&=[\alpha x^{+},y^{+}]\left\| z^{+}\right\|^{2}_{+}+[\alpha x^{-},y^{-}]\left\| z^{-}\right\|^{2}_{-}=\alpha\left([x^{+},y^{+}]\left\| z^{+}\right\|^{2}_{+}+[ x^{-},y^{-}]\left\| z^{-}\right\|_{-}\right)\\
&=\alpha\psi( x,y,z).
\end{align*}

$\protect\mycirc{3}$
\begin{align*}
\psi( y,x,z)&=[ y^{+},x^{+}]\left\| z^{+}\right\|_{+}^{2}+[y^{-},x^{-}]\left\| z^{-}\right\|_{-}^{2}=\overline{[ x^{+},y^{+}]}\left\| z^{+}\right\|_{+}^{2}+\overline{[ x^{-},y^{-}]}\left\| z^{-}\right\|_{-}^{2}\\
&=\overline{[ x^{+},y^{+}]\left\| z^{+}\right\|_{+}^{2}}+\overline{[ x^{-},y^{-}]\left\| z^{-}\right\|_{-}^{2}}=\overline{[ x^{+},y^{+}]\left\| z^{+}\right\|_{+}^{2}+[ x^{-},y^{-}]\left\| z^{-}\right\|_{-}^{2}}\\
&=\overline{\psi(x,y,z)}.
\end{align*}

$\protect\mycirc{4}$
\begin{align*}
\psi( x,x,z)&=[ x^{+},x^{+}]\left\| z^{+}\right\|_{+}^{2}+[x^{-},x^{-}]\left\| z^{-}\right\|_{-}^{2}\\
&=[ x^{+},x^{+}][z^{+},z^{+}]+[x^{-},x^{-}](-[z^{-},z^{-}])\\
&=[ x^{+},x^{+}][z^{+},z^{+}]+(-[x^{-},x^{-}])[z^{-},z^{-}]\\
&=[z^{+},z^{+}]\left\|x^{+}\right\|^{2}_{+}+[z^{-},z^{-}]\left\|x^{-}\right\|_{-}^{2}=\psi( z,z,x).
\end{align*}

$\protect\mycirc{5}$ There exist $(k^{+}, l^{-}) \in \mathcal{F}^{+} \times \mathcal{F}^{-}$ such that for any $w \in \mathcal{F}$, the following holds:
$$
\psi(k^{+},k^{+},w)=[k^{+},k^{+}]\|w^{+}\|^{2}_{+}+[\mathbf{0},\mathbf{0}]\|w^{-}\|^{2}_{-}=[k^{+},k^{+}]\|w^{+}\|^{2}_{+}\ge 0
$$
and
$$
\psi(l^{-},l^{-},w)=[\mathbf{0},\mathbf{0}]\|w^{+}\|^{2}_{+}+[l^{-},l^{-}]\|w^{-}\|^{2}_{-}=[l^{-},l^{-}]\|w^{-}\|^{2}_{-}\le 0.
$$
\end{proof}

\begin{definition}
Let $\mathcal{F}$ be a complex vector space of dimension $d \geq 2$. If $\psi: \mathcal{F} \times \mathcal{F} \times \mathcal{F} \to \mathbb{C}$ is a function that satisfies the conditions of Proposition \ref{k2k}, then we will call the pair $(\mathcal{F}, \psi)$ an indefinite two-inner product space.
\end{definition}

\begin{remark}\label{presesrvaort}
Note that for any $(k^{+}, l^{-}) \in \mathcal{F}^{+} \times \mathcal{F}^{-}$ and $w = w^{+} + w^{-} \in \mathcal{F}$, the following holds:
$$
\psi(k^{+}, l^{-} , w) = [k^{+}, \mathbf{0}] \|w^{+}\|_{+} + [\mathbf{0}, l^{-}] \|w^{-}\|_{-} = 0.
$$
Therefore, $\mathcal{F}^{+}$ and $\mathcal{F}^{-}$ preserve orthogonality in the indefinite two-inner product $\psi$.
\end{remark}

From now on, the space with the indefinite two-metric given in \eqref{two-krein ind} will be referred to as the standardized indefinite two-metric or standardized two-inner product, assuming it originates from a space with an indefinite metric $(\mathcal{F}, [\cdot, \cdot], \mathcal{J})$.

\begin{remark}\label{terceraj}
Note that given a space with a standardized indefinite two-metric, the following holds:
\begin{align*}
\psi( x,x ,\mathcal{J}t)&=\psi( x,x ,t^{+}-t^{-})=[ x^{+},x^{+}] \|t^{+}\|^{2}_{+}+[ x^{-},x^{-}] \|-t^{-}\|^{2}_{-}\\
&=[ x^{+},x^{+}] \|t^{+}\|^{2}_{+}+[ x^{-},x^{-}] \|t^{-}\|^{2}_{-}=\psi( x,x ,t).
\end{align*}
\end{remark}

\begin{remark}\label{two-normpositiveandnegative}
From Proposition \ref{k2k}, the two-norms associated with the spaces with two-inner product $(\mathcal{F}^{+}, \psi)$ and $(\mathcal{F}^{-}, -\psi)$ naturally follow. These two-norms are given as follows:
$$
\mathcal{N}^{+}(x^{+},y^{+})=\sqrt{\psi(x^{+},x^{+},y^{+})}=\sqrt{[x^{+},x^{+}]\left\| y^{+}\right\|_{+}^{2}}=\sqrt{[x^{+},x^{+}]}\sqrt{\left\| y^{+}\right\|_{+}^{2}}=\left\| x^{+}\right\|_{+}\left\| y^{+}\right\|_{+}\ge 0
$$
and
$$
\mathcal{N}^{-}(x^{-},y^{-})=\sqrt{-\psi(x^{-},x^{-},y^{-})}=\sqrt{-[x^{-},x^{-}]\left\| y^{-}\right\|_{-}^{2}}=\left\| x^{-}\right\|_{-}\left\| y^{-}\right\|_{-}\ge 0.
$$
\end{remark}

\begin{proposition}
Let $(\mathcal{F} = \mathcal{F}^{+}[\overset{\cdot}{+}]\mathcal{F}^{-}, [\cdot, \cdot], \mathcal{J})$ be a Krein space with the standardized two-inner product $\psi$ given in Proposition \ref{k2k}. The application $\psi_{\mathcal{J}}: \mathcal{F} \times \mathcal{F} \times \mathcal{F} \longrightarrow \mathbb{C}$ defined by
$$
\psi_{\mathcal{J}}(x, y , z) := \psi(\mathcal{J}x, y , z) \quad \text{for all } x, y, z \in \mathcal{F},
$$
is a positive semidefinite two-inner product, which, from now on, we will refer to as the associated $\mathcal{J}$-two-inner product to $\psi$, or simply the $\mathcal{J}$-two-inner product.
\end{proposition}
\begin{proof}
Let $x, y, z \in \mathcal{F}$. Then,
\begin{align*}
\psi_{\mathcal{J}}( x,x,z)&=\psi(\mathcal{J} x,x,z)=\psi( x^{+}-x^{-},x^{+}+x^{-},z)=[x^{+},x^{+}]\left\| z^{+}\right\|^{2}_{+}-[x^{-},x^{-}]\left\| z^{-}\right\|^{2}_{-}\\
&=[x^{+},x^{+}]\left\| z^{+}\right\|^{2}_{+}+(-[x^{-},x^{-}])\left\| z^{-}\right\|^{2}_{-}\geq 0.
\end{align*}
\end{proof}
\begin{proposition}\label{j2norm}
Let $(\mathcal{F} = \mathcal{F}^{+}[\overset{\cdot}{+}]\mathcal{F}^{-}, [\cdot, \cdot], \mathcal{J})$ be a Krein space with the standardized two-inner product $\psi$ given in Proposition \ref{k2k}. Then, the application $\mathcal{N}_{\mathcal{J}}:\mathcal{F} \times \mathcal{F}\longrightarrow\mathbb{R}$  defined by $\mathcal{N}_{\mathcal{J}}(x, y) = \sqrt{\psi_{\mathcal{J}}(x, x , y)}$ is a two-norm over $\mathcal{F}$.

\end{proposition}
\begin{proof}
Let $x,y,z\in\mathcal{F}$ and $\alpha\in\mathbb{R}$.
\begin{enumerate}
\item[$(2N_1)$] $\mathcal{N}_{\mathcal{J}}(x,\alpha y)=\sqrt{\psi_{\mathcal{J}}(x,x,\alpha y)}=\sqrt{|\alpha|^{2}\psi_{\mathcal{J}}(x,x, y)}=|\alpha|\sqrt{\psi_{\mathcal{J}}(x,x, y)}=|\alpha|\mathcal{N}_{\mathcal{J}}(x,y).$

\item[$(2N_2)$] Given that $\psi_{\mathcal{J}}$ is a semi-positive definite two-inner product, by Proposition \ref{2schw} it follows that
\begin{align*}
\mathcal{N}_{\mathcal{J}}(x,y+z)
 &=\sqrt{\psi_{\mathcal{J}}(x,x, y) }=\sqrt{\psi_{\mathcal{J}}(x,x, y)+\psi_{\mathcal{J}}(x,x, z)+2Re(\psi_{\mathcal{J}}(x,y, z))} \\
 &\leq \sqrt{\psi_{\mathcal{J}}(x,x, y)+\psi_{\mathcal{J}}(x,x, z)+2\left|\psi_{\mathcal{J}}(x,y, z)\right|} \\
 &\leq \sqrt{\mathcal{N}_{\mathcal{J}}(x,y)^{2}+\mathcal{N}_{\mathcal{J}}(x,z)^{2}+2\mathcal{N}_{\mathcal{J}}(x,z)\mathcal{N}_{\mathcal{J}}(y,z)} \\
 &=\sqrt{\left( \mathcal{N}_{\mathcal{J}}(x,z)+\mathcal{N}_{\mathcal{J}}(y,z)\right)^{2}} =\mathcal{N}_{\mathcal{J}}(x,z)+\mathcal{N}_{\mathcal{J}}(y,z).
\end{align*}
 
\item[$(2N_3)$] $\mathcal{N}_{\mathcal{J}}(x,y)=\sqrt{\psi_{\mathcal{J}}(x,x, y)}=\sqrt{\psi_{\mathcal{J}}(y,y, x)}=\mathcal{N}_{\mathcal{J}}(y,x).$
\end{enumerate}

We will refer to this two-norm simply as $\mathcal{J}$-two-norm.

\end{proof}

\begin{remark}
Note that the two-norms induced in the spaces with two-inner product $(\mathcal{F}^{+},\psi)$ and $(\mathcal{F}^{-},-\psi)$ are nothing more than a restriction of the $\mathcal{J}$-two-norm to these spaces.
$$
\mathcal{N}_{\mathcal{J}}(x^{+},y^{+}):
=\sqrt{\psi_{\mathcal{J}}(x^{+},x^{+},y^{+})_{\mathcal {J}}}=\sqrt{\psi(\mathcal Jx^{+},x^{+},y^{+})}=\sqrt{\psi(x^{+},x^{+},y^{+})}=\mathcal{N}^{{+}}(x^{+},y^{+})
$$
and
$$
\mathcal{N}_{\mathcal{J}}(x^{-},y^{-})=\sqrt{\psi_{\mathcal{J}}(x^{-},x^{-},y^{-})_{\mathcal{J}}}=\sqrt{\psi(\mathcal{J}x^{-},x^{-},y^{-})}=\sqrt{-\psi(x^{-},x^{-},y^{-})}=\mathcal{N}^{-}(x^{-},y^{-}),
$$
where $x^{+},y^{+}\in\mathcal{F}^{+}\text{ and }x^{-},y^{-}\in\mathcal{F}^{-}.$
\end{remark}

\begin{theorem}\label{Desigualdades de normas}
Let $(\mathcal{F} = \mathcal{F}^{+} \dot{[+]}\mathcal{F}^{-}, \psi, \mathcal{J})$ be a space with a standardized indefinite two-metric. Then,
$$\mathcal{N}_{\mathcal{J}}(x,z)\le \mathcal{N}^{+}( x^+ ,z^+ )+\mathcal{N}^{-}( x^- ,z^{-}),$$
for all $x=x^{+}+x^{-}\in\mathcal{F}$.
\end{theorem}

\begin{proof}
\begin{align*}
\mathcal{N}_{\mathcal{J}}(x,z)^2&=\psi_{\mathcal{J}}(x,x, z)=\psi(\mathcal{J}x,x,z)=\psi(x^+-x^-,x^+ +x^-,z^+ +z^-)\\
&=\psi(x^+,x^+,z^+)-\psi(x^-,x^-,z^-)=\psi(x^+,x^+,z^+)+(-\psi(x^-,x^-,z^-))\\
&=(\sqrt{\psi(x^+,x^+,z^+)})^2+(\sqrt{-\psi(x^-,x^-,z^-)})^{2}=\mathcal{N}^{+}(x^+,z^+)^2+\mathcal{N}^{-}(x^-,z^-)^2 \\
&\le \mathcal{N}^{+}(x^+,z^+)^2+2\mathcal{N}^{+}(x^+,z^+)\mathcal{N}^{-}(x^-,z^-) +\mathcal{N}^{-}(x^-,z^-)^2 \\
&=(\mathcal{N}^{+}(x^+,z^+)+\mathcal{N}^{-}(x^-,z^-) )^2.
\end{align*}
Therefore,
$$\mathcal{N}_{\mathcal{J}}(x,z)\le \mathcal{N}^{+}( x^+ ,z^+)+\mathcal{N}^{-}( x^-,z^- ) \text{ for each } x=x^++x^-,z=z^++z^- \in\mathcal{F}$$
\end{proof}

\begin{theorem}\label{2equiv}
Let $(\mathcal{F}, [\cdot , \cdot ])$ be a Krein space with fundamental decompositions $\mathcal{F} = \mathcal{F}^{+}_{1} \dot{[+]} \mathcal{F}^{-}_{1}$, $\mathcal{F} = \mathcal{F}^{+}_{2} \dot{[+]} \mathcal{F}^{-}_{2},$
and fundamental symmetries $\mathcal{J}_{1}$ and $\mathcal{J}_{2}$ respectively. Then, the two-norms $\mathcal{N}_{\mathcal{J}_{1}}$ and $\mathcal{N}_{\mathcal{J}_{2}}$ associated to the standardized two-Krein space are equivalent.
\end{theorem}
\begin{proof}
Let $x, y \in \mathcal{F}$. By Theorem \ref{equiv}, there exist $\alpha, \beta \in \mathbb{R}^{+}$ such that
$$
\alpha\|x\|_{\mathcal{J}_{1}}\le\|x\|_{\mathcal{J}_{2}}\le\beta \|x\|_{\mathcal{J}_{1}}.
$$
In this way, we obtain on one hand that
\begin{align*}
\mathcal{N}_{\mathcal{J}_{1}}(x,z)^{2}&=\psi_{\mathcal{J}_{1}}(x,x, z)=\|x^{+}\|_{+}^{2}\left\| z^{+}\right\|^{2}_{+}+\|x^{-}\|_{-}^{2}\left\| z^{-}\right\|^{2}_{-}\\
&=\|x^{+}\|_{\mathcal{J}_{1}}^{2}\left\| z^{+}\right\|^{2}_{\mathcal{J}_{1}}+\|x^{-}\|_{\mathcal{J}_{1}}^{2}\left\| z^{-}\right\|^{2}_{\mathcal{J}_{1}}\le \frac{1}{\alpha^{2}} ( \|x^{+}\|_{\mathcal{J}_{2}}^{2}\left\| z^{+}\right\|^{2}_{\mathcal{J}_{2}}+\|x^{-}\|_{\mathcal{J}_{2}}^{2}\left\| z^{-}\right\|^{2}_{\mathcal{J}_{2}})\\
&=\frac{1}{\alpha^{2}}\mathcal{N}_{\mathcal{J}_{2}}(x,z)^{2}
\end{align*}
Similarly, it follows that
\begin{align*}
\mathcal{N}_{\mathcal{J}_{2}}(x,z)^{2}&=\psi_{\mathcal{J}_{2}}(x,x, z)=\|x^{+}\|_{+}^{2}\left\| z^{+}\right\|^{2}_{+}+\|x^{-}\|_{-}^{2}\left\| z^{-}\right\|^{2}_{-}\\
&=\|x^{+}\|_{\mathcal{J}_{2}}^{2}\left\| z^{+}\right\|^{2}_{\mathcal{J}_{2}}+\|x^{-}\|_{\mathcal{J}_{2}}^{2}\left\| z^{-}\right\|^{2}_{\mathcal{J}_{2}}\le\beta^{2}\mathcal{N}_{\mathcal{J}_{1}}(x,z)^{2}.
\end{align*}
Therefore, it holds that $\alpha \mathcal{N}_{\mathcal{J}_{1}}(x,z)\le\mathcal{N}_{\mathcal{J}_{2}}(x,z)\le \beta \mathcal{N}_{\mathcal{J}_{1}}(x,z)$.
\end{proof}

\begin{proposition}\label{completez}
Let $(\mathcal{F}=\mathcal{F}^{+}[\overset{\cdot}{+}]\mathcal{F}^{-},[\cdot,\cdot ],\mathcal{J})$ be a Krein space with the standardized two-inner product defined by $\psi(x,y, z)=[x^{+},y^{+}]\left\| z^{+}\right\|^{2}_{+}+[x^{-},y^{-}]\left\| z^{-}\right\|^{2}_{-}$ and $t\in\mathcal{F}$. Then, $(\mathcal{F}^{+},\psi)$ and $(\mathcal{F}^{-},-\psi)$ are $t$-Hilbert spaces.
\end{proposition}
\begin{proof}
Let $\{x^{+}_{n}\}_{n\in\mathbb{N}}$ be a Cauchy sequence in $(\mathcal{F}^{+},\psi)$, $t=t^{+}+t^{-}\neq 0\in\mathcal{F}$ fixed, and $\epsilon>0$ given. Take $\epsilon^{*}=\|t^{+}\|_{+}\epsilon >0$. Then, there exists $N\in\mathbb{N}$ such that for all $t^{+}\in\mathcal{F}^{+}$, it holds that $$\mathcal{N}^{+}(x^{+}_{m}-x^{+}_{n},t^{+})
<\epsilon^{*},$$ whenever $m,n>N$. Thus,
$$
\mathcal{N}^{+}(x^{+}_{m}-x^{+}_{n},t^{+})=\sqrt{\psi (x^{+}_{m}-x^{+}_{n},x^{+}_{m}-x^{+}_{n},t^{+})}=\sqrt{\|x^{+}_{m}-x^{+}_{n}\|^{2}_{+}\|t^{+}\|^{2}_{+}}<\epsilon^{*} ,
$$
hence, we have that
$$
\|x^{+}_{m}-x^{+}_{n}\|_{+}<\frac{1}{\|t^{+}\|_{+}}\epsilon^{*}=\epsilon .
$$
Consequently, $\{x^{+}_{n}\}_{n\in\mathbb{N}}$ is a Cauchy sequence in $(\mathcal{F}^{+},[\cdot,\cdot])$. Taking $\epsilon'=\frac{1}{\|t^{+}\|_{+}}\epsilon$ and by the completeness of $(\mathcal{F}^{+},[\cdot,\cdot])$, it follows that there exist $x^{+}\in\mathcal{F}^{+}$, $N'\in\mathbb{N}$, such that if $n>N'_{2}$ then $\|x^{+}_{n}-x^{+}\|< \epsilon'$.

Then,
$$
\mathcal{N}^{+}(x^{+}_{n}-x^{+},t^{+})=\sqrt{\|x^{+}_{n}-x^{+}\|^{2}_{+}\|t^{+}\|^{2}_{+}}<\epsilon^{'} \|t^{+}\|=\epsilon .
$$
In this way, $\{x_{n}\}_{n\in\mathbb{N}}$ converges to $x^{+}\in\mathcal{F}^{+}$. Therefore, $(\mathcal{F}^{+},\psi)$ is a $t$-Hilbert space.

Using an analogous reasoning, it can be seen that $(\mathcal{F}^{-},-\psi)$ is a $t$-Hilbert space.
\end{proof}

\begin{remark}
A vector space $\mathcal{F}$ endowed with a mapping $\psi\colon \mathcal{F}\times \mathcal{F}\times \mathcal{F}\longrightarrow \mathbb{C} $ that satisfies the characteristics of Proposition \ref{k2k} motivates us to give the following definition.
\end{remark}

\begin{definition}\label{two-krein}
A space with an indefinite two-inner product $(\mathcal{F},\psi )$ that admits a fundamental decomposition of the form $\mathcal{F} = \mathcal{F}^{+} \dot{[+]} \mathcal{F}^{-}, $ such that $(\mathcal{F}^{+}, \psi)$ and $(\mathcal{F}^{-}, - \psi)$ are two-Hilbert spaces will be called a \textbf{two-Krein space}.
\end{definition}

\begin{remark}
Let $(\mathcal{F}=\mathcal{F}^{+} \dot{[+]} \mathcal{F}^{-},[\cdot,\cdot],\mathcal{J} )$ be a Krein space, $[a,b]$ an interval, and \\$f:[a,b]\to \mathcal{F}=\mathcal{F}^{+}\dot{[+]} \mathcal{F}^{-}$. Taking into account that for any $t$ in $[a,b]$, $f(t)$ belongs to\\ $\mathcal{F}=\mathcal{F}^{+} \dot{[+]}\mathcal{F}^{+}$, from now on we will write the image of $t$ under $f $ as $f(t)=f^{+}(t)+f^{-}(t)$.
\end{remark}

\begin{example} \label{R2Krein}
Let us consider the set formed by pairs of complex numbers $\mathbb{C}^{2}$ endowed with the indefinite inner product $[\cdot , \cdot ] : \mathbb{C}^{2} \times \mathbb{C}^{2} \longrightarrow \mathbb{R}$ given by
\begin{eqnarray*}
[(a, b), (c, d)] := a\overline{c} - b\overline{d}, \, \, \,  (a, b), (c, d) \in \mathbb{C}^{2},
\end{eqnarray*}
which is a Krein space with fundamental symmetry $\mathcal{J}$ given by $$\mathcal{J}(a,b)=(a,0)-(0,b)=(a,-b).$$

Using \eqref{two-krein ind}, we have that the standardized indefinite two-inner product induced by this indefinite metric is given by the function $\psi\colon \mathbb{C}^{2}\times\mathbb{C}^{2}\times\mathbb{C}^{2}\longrightarrow \mathbb{C}$ defined by
\begin{align*}
\psi((x_{1},x_{2}),(y_{1},y_{2}),(z_{1},z_{2}))&=[(x_{1},0),(y_{1},0)]\|(z_{1},0)\|^{2}_{+}+[(0,x_{2}),(0,y_{2})]\|(0,z_{2})\|^{2}_{-}\\
&=x_{1}\overline{y_{1}}|z_{1}|^{2}-x_{2}\overline{y_{2}}|z_{2}|^{2}.
\end{align*}
Let
$$
\mathbb{C}^{2}{}^{+} := \left\lbrace (x, 0) | x \in \mathbb{C} \right\rbrace \quad \textup{and} \quad \mathbb{C}^{2}{}^{-} := \left\lbrace (0, y) | y \in \mathbb{C} \right\rbrace .
$$

Note that for all $(t_{1},t_{2})\in\mathbb{C}^{2}$ we have
$$
\psi((x,0),(x,0),(t_{1},t_{2}))=x\overline{x}|t_{1}|^{2}-0\cdot\overline{0}|t_{2}|^{2}=|x|^{2}|t_{1}|^{2}\geq 0
$$
and
$$
\psi((0,x),(0,x),(t_{1},t_{2}))=0\cdot\overline{0}|t_{1}|^{2}-x\overline{x}|t_{2}|^{2}=-|x|^{2}|t_{2}|^{2}\leq 0.
$$
Also,
$$
\psi((x,0),(0,y),(t_{1},t_{2}))=x\cdot\overline{0}|t_{1}|^{2}-0\cdot\overline{y}|t_{2}|^{2}=0.
$$
Furthermore, $(\mathbb{C}^{2}{}^{+},\psi)$ and $(\mathbb{C}^{2}{}^{-},-\psi)$ are two-Hilbert spaces. Therefore, $(\mathcal{F}=\mathcal{F}^{+}[\overset{\cdot}{+}]\mathcal{F}^{-},\psi)$ is a two-Krein space.

On the other hand, the $\mathcal{J}$-two-inner product is determined by the expression
\begin{align*}
\psi_{\mathcal{J}}((x_{1},x_{2}),(y_{1},y_{2}),(z_{1},z_{2}))&=\psi(\mathcal{J}(x_{1},x_{2}),(y_{1},y_{2}),(z_{1},z_{2}))=\psi((x_{1},-x_{2}),(y_{1},y_{2}),(z_{1},z_{2}))\\
&=x_{1}\overline{y_{1}}|z_{1}|^{2}+x_{2}\overline{y_{2}}|z_{2}|^{2}.
\end{align*}
\end{example}

\section{Transfer of properties from the fundamental symmetry to the fundamental two-symmetry}

\begin{proposition}\label{propiedades j}
Let $(\mathcal{F}=\mathcal{F}^{+}[\overset{\cdot}{+}]\mathcal{F}^{-},[\cdot,\cdot ],\mathcal{J})$ be a Krein space with the associated two-inner product defined by $\psi(x,y, z)=[x^{+},y^{+}]\left\| z\right\|^{2}_{+}+[x^{-},y^{-}]\left\| z\right\|^{2}_{-}$. Then, $\mathcal{J}$ preserves the following properties:
\begin{enumerate}
\item $\mathcal{J}$ is two-symmetric and two-self-adjoint.
\item $\mathcal{J}$ is two-isometric.
\item $\mathcal{J}$ is two-bounded.
\end{enumerate}
\end{proposition}
\begin{proof}
1.
\begin{align*}
\psi(\mathcal{J}x,y,z)&=\psi(x^{+}-x^{-},y,z)=[x^{+},y^{+}]\|z^{+}\|^{2}_{+}-[x^{-},y^{-}]\|z^{-}\|^{2}_{-}\\
&=[x^{+},y^{+}]\|z^{+}\|^{2}_{+}+[x^{-},-y^{-}]\|z^{-}\|^{2}_{-}=\psi(x,y^{+}-y^{-},z)=\psi(x,\mathcal{J}y,z).
\end{align*}

2. 
\begin{align*}
\psi(\mathcal{J}x,\mathcal{J}y,z)&=\psi(x^{+}-x^{-},y^{+}-y^{-},z)=[x^{+},y^{+}]\|z^{+}\|^{2}_{+}+[-x^{-},-y^{-}]\|z^{-}\|^{2}_{-}\\
&=[x^{+},y^{+}]\|z^{+}\|^{2}_{+}+[x^{-},y^{-}]\|z^{-}\|^{2}_{-}=\psi(x,y, z).
\end{align*}

3. Taking into account Remark \ref{terceraj}, it follows that
\begin{align*}
 \mathcal{N}\left(\mathcal{J} x,t\right) + \mathcal{N}(x,\mathcal{J}t)&=\sqrt{\psi( \mathcal{J}x,\mathcal{J}x ,t)}+\sqrt{\psi( x,x ,\mathcal{J}t)}\\
 &=\sqrt{\psi( x,x ,t)}+\sqrt{\psi( x,x ,t^{+}-t^{-})}\\
 &=\sqrt{\psi( x,x ,t)}+\sqrt{\psi( x,x ,t)}=2\sqrt{\psi( x,x ,t)}\\
 &\le 3\mathcal{N}(x,t).
\end{align*}
\end{proof}

\section{Variation of a function in spaces with an indefinite two-metric}

Next, we introduce the notion of a function of bounded variation in spaces with an indefinite two-metric.

\begin{definition}\label{tvar}
Let $({\mathcal{F}} =\mathcal{F}^{+} \dot{[+]} \mathcal{F}^{-},{\psi},\mathcal{J})$ be a standardized two-Krein space and $t\in \mathcal{F}$ fixed. Given $[a,b]$ a closed interval, $P=\{a=t_{0}, t_{1},\cdots b=t_{n}\}\in\mathcal{P}([a,b])$ a partition and $f\colon [a,b]\to \mathcal{F}$ a function, we define the non-negative number
$$
V_{a}^{b}(f,\mathcal{F},t;P)_{\mathcal{J}}=\sum_{i=1}^{n}  \mathcal{N}_{\mathcal{J}}(f(t_i)-f(t_{i-1}),t)
$$
as the \emph{$t$-variation of $f$ over $[a,b]$ with respect to $P$.}

Furthermore, we will call the \emph{$t$-variation of $f$ over $[a,b]$} the supremum
$$
V_a^b(f,\mathcal{F},t )_{\mathcal{J}}:=\sup \left\{ \sum_{i=1}^{n} \mathcal{N}_{\mathcal{J}}(f(t_i)-f(t_{i-1}),t)\colon P=\{x_0,x_1,\cdots,x_n\}\in\mathcal{P}[a,b]\right\}.
$$    
\end{definition}

\begin{definition}
Let $({\mathcal{F}} =\mathcal{F}^{+} \dot{[+]} \mathcal{F}^{-},{\psi},\mathcal{J})$ be a standardized two-Krein space and $t\in \mathcal{F}$ fixed. We will say that a function $f\colon [a,b]\longrightarrow \mathcal{F}$ is strongly of bounded $t$-variation over $[a,b]$ if the supremum
$$V_a^b(f,\mathcal{F},t)_{\mathcal{J}}=\sup\left\{\sum_{i=1}^n \mathcal{N}_{\mathcal{J}}(f(t_i)-f(t_{i-1}),t):P\in\mathcal{P}[a,b]\right\}$$ is finite.
\end{definition}

\begin{remark}\label{2h2k}
Note that, when $\mathcal{J}=I$, we are in the case of a two-Hilbert space, which guarantees that every function of bounded $(2,h)$-variation in a two-Hilbert space is a strongly bounded $t$-variation function in the standardized two-Krein space.
\end{remark}

Next, we show the robustness of the definition of bounded variation function in standardized two-Krein spaces. To do this, we show that this definition is independent of the fundamental decomposition.
\begin{theorem}\label{varequiv}
Let $(\mathcal{F},\psi)$ be a standardized two-Krein space with fundamental decompositions $\mathcal{F}=\mathcal{F}^{+}_{1}+\mathcal{F}^{-}_{1}$ and $\mathcal{F}=\mathcal{F}_{2}^{+}+\mathcal{F}^{-}_{2}$, and fundamental symmetries $\mathcal{J}_{1}$ and $\mathcal{J}_{2}$, respectively. If $f\colon [a,b]\longrightarrow \mathcal{F}$ is strongly of bounded $t$-variation in $(\mathcal{F},\psi,\mathcal{J}_{1})$, then it is strongly of bounded $t$-variation in $(\mathcal{F},\psi,\mathcal{J}_{2})$.
\end{theorem}
\begin{proof}
By Theorem \ref{2equiv}, there exist $\alpha,\beta\in\mathbb{R}^{+}$ such that
$$
\alpha \mathcal{N}_{\mathcal{J}_{1}}(x,z)\le\mathcal{N}_{\mathcal{J}_{2}}(x,z)\le\beta\mathcal{N}_{\mathcal{J}_{1}}(x,z), \quad\text{for all $x,z\in\mathcal{F}$.}
$$

Furthermore, since $f$ is strongly of bounded $t$-variation in $(\mathcal{F},\psi,\mathcal{J}_{1})$, there exists $c\in\mathbb{R}^{+}$ such that $V_{a}^{b}(f,\mathcal{F},t)_{\mathcal{J}_{1}}\le c$.

Taking $\frac{1}{\alpha}c=l>0$, we obtain
$$
V_{a}^{b}(f,\mathcal{F},t)_{\mathcal{J}_{2}}\le l.
$$
\end{proof}

\begin{example}
Consider the Krein space $(\mathbb{C}^{2},[\cdot ,\cdot],\mathcal{J})$ given in Example \ref{R2Krein}, whose standardized indefinite two-inner product $\psi$ is given by $\psi (a,b,c)=a_{1}\overline{b_{1}}|c_{1}|^{2}-a_{2}\overline{b_{2}}|c_{2}|^{2}$ for $a,b,c\in \mathbb{C}^{2}$ and the application $f\colon [a,b]\longrightarrow\mathbb{C}^{2}$ defined by $f(x)=(-xi,i)$. Take $t=(1,0)\in\mathbb{C}^{2}$. Let's see that the $(1,0)$-variation of $f$ over $[a,b]$ in $\mathbb{C}^{2}$ is finite. Indeed, let $P=\{a,\cdots x_{n-1},b\}\in\mathcal{P}([a,b])$.
First,
$$
f(x_j)-f(x_{j-1})=(-ix_{j},i)-(-ix_{j-1},i)=(-i(x_{j}-x_{j-1}),0)
$$
Then:
\begin{align*}
\sum_{j=1}^{n}  \mathcal{N}_{\mathcal{J}}( (-i(x_{j}-x_{j-1}),0), (1,0))&=\sum_{j=1}^{n}\sqrt{|-i(x_{j}-x_{j-1})|^{2} |1|^{2}}=\sum_{j=1}^{n}|(x_{j}-x_{j-1})|\\
&=\sum_{j=1}^{n}(x_{j}-x_{j-1})=b-a
\end{align*}
Thus, there exists $\alpha =b-a>0$ such that
$$
V_{a}^{b}(f,\mathbb{C}^{2},(1,0))_{\mathcal{J}}=\sup_{P\in\mathcal{P}([a,b])} \left\{ \sum_{j=1}^{n}  \mathcal{N}_{\mathcal{J}}( f(x_i)-f(x_{i-1}), (1,0))\right\}\le \alpha .
$$
Therefore, $f$ has bounded $(1,0)$-variation over $[a,b]$ in the standardized two-Krein space $(\mathbb{C}^{2},\psi,\mathcal{J})$.
\end{example}

\begin{remark}
The class of strongly bounded $t$-variation functions over $[a,b]$ in a standardized two-Krein space $({\mathcal{F}} =\mathcal{F}^{+} \dot{[+]} \mathcal{F}^{-},{\psi},\mathcal{J})$ will be denoted by $BV([a,b],(\mathcal{F},\psi_{\text{stand}}))$.
\end{remark}

\begin{theorem}
Let $({\mathcal{F}} =\mathcal{F}^{+} \dot{[+]} \mathcal{F}^{-},{\psi},\mathcal{J})$ be a standardized two-Krein space and $t=t^{+}+t^{-}\in \mathcal{F}$ fixed. If $f:[a,b]\to\mathcal{F}$ is strongly of bounded $t$-variation over $[a,b]$ in $(\mathcal{F} = \mathcal{F}^{+} \dot{[+]} \mathcal{F}^{-})$, then $f$ is strongly of bounded $t$-variation in the two-Hilbert subspaces $(\mathcal{F}^{+},\psi)$ and $(\mathcal{F}^{-},-\psi)$.
\end{theorem} 
\begin{proof}
Let $P=\{a,x_1\cdots x_{n-1},b\}\in\mathcal{P}([a,b])$ be a partition. Since $f:[a,b]\to\mathcal{F}$ is strongly of bounded $t$-variation over $[a,b]$ in $(\mathcal{F} = \mathcal{F}^{+} \dot{[+]} \mathcal{F}^{-},\psi,\mathcal{J})$, then there exists $\alpha\in\mathbb{R}^{+}$ such that
$$
V_a^b(f,\mathcal{F},t )_{\mathcal{J}}:=\sup_{P\in\mathcal{P}([a,b])} \left\{ \sum_{i=1}^{n}  \mathcal{N}_{\mathcal{J}}(f(x_i)-f(x_{i-1}), t )\right\}\le \alpha .
$$
Furthermore,
\begin{align*}
\mathcal{N}_{\mathcal{J}}(f(x_i)-f(x_{i-1}), t )&=\sqrt{\|f^{+}(x_i)-f^{+}(x_{i-1})\|^{2}_{+}\left\| t^{+}\right\|^{2}_{+}+\|f^{-}(x_i)-f^{-}(x_{i-1})\|^{2}_{-}\left\| t^{-}\right\|^{2}_{-}}.
\end{align*}
Note that by Remark \ref{two-normpositiveandnegative} it holds that
$$
\mathcal{N}_{\mathcal{J}}(f(x_i)-f(x_{i-1}), t )\ge \|f^{+}(x_i)-f^{+}(x_{i-1})\|_{+}\left\| t^{+}\right\|_{+},\|f^{-}(x_i)-f^{-}(x_{i-1})\|_{-}\left\| t^{-}\right\|_{-}.
$$
That is,
$$
 \mathcal{N}^{+}( f^{+}(x_i)-f^{+}(x_{i-1}), t), \mathcal{N}^{-}( f^{-}(x_i)-f^{-}(x_{i-1}), t)\le  \mathcal{N}_{\mathcal{J}}(f(t_i)-f(t_{i-1}),t).
$$
Therefore:
$$
\overset{+}{V^{b}_{a}}(f,\mathcal{F},t )_{\mathcal{J}}:=\sup_{P\in\mathcal{P}([a,b])} \left\{ \sum_{i=1}^{n}  \mathcal{N}^{+}( f^{+}(x_i)-f^{+}(x_{i-1}), t)\right\}\le \alpha
$$
and
$$
\overset{-}{V^{b}_{a}}(f,\mathcal{F},t )_{\mathcal{J}}:=\sup_{P\in\mathcal{P}([a,b])} \left\{ \sum_{i=1}^{n}  \mathcal{N}^{-}( f^{-}(x_i)-f^{-}(x_{i-1}), t)\right\}\le \alpha .
$$
\end{proof}

Next, we show that the $t$-variation in a space with an indefinite two-metric \ref{tvar} extends the notion presented in Definition \ref{fuertemente de variacion acotada}.
\begin{proposition}\label{va implica 2kva}
Let $({\mathcal{F}} =\mathcal{F}^{+} \dot{[+]} \mathcal{F}^{-},{[\cdot , \cdot ]},\mathcal{J})$ be a Krein space and $t\in\mathcal{F}$. If
$f\colon [a,b]\to \mathcal{F}$ is a strongly bounded variation function in $\left(\mathcal{F},[ \cdot,\cdot ]\right)$, then $f$ is strongly bounded $t$-variation in the standardized two-Krein space induced by $[\cdot,\cdot]$.
\end{proposition}
\begin{proof}
Let $P=\{a,x_1,\cdots,x_{n-1},b\}\in\mathcal{P}[a,b]$. Since $f$ is strongly of bounded variation in the classical Krein space $\left(\mathcal{F},[ \cdot,\cdot ]\right)$, there exists $\alpha\in \mathbb{R}^+$ such that
$$
V_a^b(f,\mathcal{F})=\sup\left\{\sum_{i=1}^n(\|f^{+}(t_i)-f^{+}(t_{i-1})\|_{+}+\|f^{-}(t_i)-f^{-}(t_{i-1})\|_{-})\right\}\le \alpha.
$$
Using Theorem \ref{Desigualdades de normas}, it follows that
$$
\mathcal{N}_{\mathcal{J}}(f(x_i)-f(x_{i-1}), t )\le \|f^{+}(x_i)-f^{+}(x_{i-1})\|_{+}\left\| t^{+}\right\|_{+}+\|f^{-}(x_i)-f^{-}(x_{i-1})\|_{-}\left\| t^{-}\right\|_{-}.
$$
Furthermore,
\begin{align*}
\mathcal{N}_{\mathcal{J}}(f(x_i)-f(x_{i-1}), t )&\le \|f^{+}(x_i)-f^{+}(x_{i-1})\|_{+}(\left\| t^{+}\right\|_{+}+\left\| t^{-}\right\|_{-})\\
&\ \ \ +\|f^{-}(x_i)-f^{-}(x_{i-1})\|_{-}(\left\| t^{+}\right\|_{+}+\left\| t^{-}\right\|_{-})\\
&=(\left\| t^{+}\right\|_{+}+\left\| t^{-}\right\|_{-})(\|f^{+}(x_i)-f^{+}(x_{i-1})\|_{+}+\|f^{-}(x_i)-f^{-}(x_{i-1})\|_{-})
\end{align*}
Thus, there exists $\beta =(\left\| t^{+}\right\|_{+}+\left\| t^{-}\right\|_{-})\alpha\ge 0$ such that it holds that
$$
V_{a}^{b}(f,\mathcal{F},t)_{\mathcal{J}}=\sup_{P\in\mathcal{P}([a,b])} \left\{ \sum_{i=1}^{n}  \mathcal{N}_{\mathcal{J}}(f(x_i)-f(x_{i-1}), t )\right\}\leq (\left\| t^{+}\right\|_{+}+\left\| t^{-}\right\|_{-})V_a^b(f,\mathcal{F})\leq \beta
$$
Therefore, $f$ is of bounded $t$-variation over $[a,b]$ in the standardized two-Krein space generated by $({\mathcal{F}} =\mathcal{F}^{+} \dot{[+]} \mathcal{F}^{-},{[\cdot , \cdot ]},\mathcal{J})$.
\end{proof}

\begin{proposition}\label{f cero constante}
Let $(\mathcal{F},\psi_{\text{stand}},\mathcal{J})$ be a standardized two-Krein space, $t\in \mathcal{F}$, and $f\colon [a,b]\to \mathcal{F}$ a strongly bounded $t$-variation function over $[a,b]$. If $V_{a}^{b}(f,\mathcal{F},t)_{\mathcal{J}}=0$, then for all $x$ in $[a,b]$ it holds that $\mathcal{N}_{\mathcal{J}}(f(x),t)=\mathcal{N}_{\mathcal{J}}(f(a),t)=\mathcal{N}_{\mathcal{J}}(f(b),t)$.
\end{proposition}
\begin{proof}
Suppose that $V_a^b(f,\mathcal{F},t)_{\mathcal{J}}=0$, that is, $$V_a^b(f,\mathcal{F},t)_{\mathcal{J}}=\sup\Bigg\{\sum_{i=1}^{n}\mathcal{N}_{\mathcal{J}}(f(x_i)-f(x_{i-1}), t ):P\in\mathcal{P}[a,b]\Bigg\}=0.$$
Then, for any partition of $[a,b]$, it holds that
\begin{equation}\label{hihi}
\sum_{i=1}^{n}\mathcal{N}_{\mathcal{J}}(f(x_i)-f(x_{i-1}), t )=0.  
\end{equation}
Let $x\in(a,b)$. Consider in particular the partition $P=\{a,x,b\}$ of $[a,b]$. By (\ref{hihi}), it holds that
$$
\mathcal{N}_{\mathcal{J}}( f(x)-f(a), t)+\mathcal{N}_{\mathcal{J}}( f(b)-f(x), t)=0.
$$
Thus,
$$
\big|\mathcal{N}_{\mathcal{J}}( f(x), t)-\mathcal{N}_{\mathcal{J}}( f(a), t)\big|\le \mathcal{N}_{\mathcal{J}}( f(x)-f(a), t)=0
$$
and
$$
\big|\mathcal{N}_{\mathcal{J}}( f(b), t)-\mathcal{N}_{\mathcal{J}}( f(x), t) \big|\le\mathcal{N}_{\mathcal{J}}( f(b)-f(x), t)=0.
$$
From this we obtain that
$$\mathcal{N}_{\mathcal{J}}(f(x),t)=\mathcal{N}_{\mathcal{J}}(f(a),t)=\mathcal{N}_{\mathcal{J}}(f(b),t)\quad\text{for all }x\in [a,b].$$
\end{proof}

The following result shows that $BV([a,b],(\mathcal{F},\psi_{\text{stand}}))$ is a subset of the $(2,t)$-bounded functions.
\begin{theorem}
Let $(\mathcal{F},\psi_{\text{stand}},\mathcal{J})$ be a standardized two-Krein space and $t\in \mathcal{F}$.\\ If $f\colon [a,b]\to \mathcal{F}$ is a strongly bounded $t$-variation function, then $f$ is $(2,t)$-bounded.
\end{theorem}
\begin{proof}
Taking into account that the application $\mathcal{N}_{\mathcal{J}}\colon \mathcal{F}\times\mathcal{F}\longrightarrow [0,\infty )$ defined by $$\mathcal{N}_{\mathcal{J}}  (x,y)=\sqrt{\psi_{\mathcal{J}}(x,x, y)} $$ is a two-norm, the proof follows from \cite{FCF}.
\end{proof}

\begin{theorem}\label{desigualdades}
Let $f,g\in BV([a,b],(\mathcal{F},\psi_{\text{stand}}))$ and $\alpha\in\mathbb{C}$. Then, the following propositions hold:
\begin{enumerate}
\item Homogeneity of variation: $V_{a}^{b}(\alpha f,\mathcal{F},t)_{\mathcal{J}}=|\alpha | V_{a}^{b}( f,\mathcal{F},t)_{\mathcal{J}}$.
\item Subadditivity of variation: $V_{a}^{b}(f+g,\mathcal{F},t)_{\mathcal{J}}\le V_{a}^{b}( f,\mathcal{F},t)_{\mathcal{J}}+V_{a}^{b}( g,\mathcal{F},t)_{\mathcal{J}}\quad$ and $\quad V_{a}^{b}(f,\mathcal{F},t+v)_{\mathcal{J}}\le V_{a}^{b}( f,\mathcal{F},t)_{\mathcal{J}}+V_{a}^{b}( f,\mathcal{F},v)_{\mathcal{J}}$.
\item Decomposition of variation: $V_{a}^{b}(f,\mathcal{F},t)_{\mathcal{J}}\le \overset{+}{V}{}_{a}^{b}( f^{+},\mathcal{F}^{+},t^{+})_{\mathcal{J}} + \overset{-}{V}{}_{a}^{b}( f^{-},\mathcal{F}^{-},t^{-})_{\mathcal{J}}$
\end{enumerate}
\end{theorem}
\begin{proof}
1.
\begin{align*}
V_{a}^{b}(\alpha f,\mathcal{F},t)_{\mathcal{J}}&=\sup_{P\in\mathcal{P}[a,b]}\left\{ \sum_{i=1}^{n}\mathcal{N}_{\mathcal{J}}(\alpha f(x_{i})-\alpha f(x_{i-1}),t) \right\}\\
&=|\alpha|\sup_{P\in\mathcal{P}[a,b]}\left\{ \sum_{i=1}^{n}\mathcal{N}_{\mathcal{J}}( f(x_{i})- f(x_{i-1}),t) \right\}=|\alpha | V_{a}^{b}( f,\mathcal{F},t)_{\mathcal{J}}
\end{align*}

2.
\begin{align*}
V_a^b(f,\mathcal{F},t)_{\mathcal{J}}+V_a^b(g,\mathcal{F},t)_{\mathcal{J}}&=\sup\left\{ \sum_{i=1}^{n}\mathcal{N}_{\mathcal{J}}(f(t_i)-f(t_{i-1}),t)+\mathcal{N}_{\mathcal{J}}(g(t_i)-g(t_{i-1}),t)\right\}\\
&\ge\sup\left\{ \sum_{i=1}^{n}\mathcal{N}_{\mathcal{J}}((f(t_i)+g(t_i))-(f(t_{i-1})+g(t_{i-1})),t)\right\}\\
&=V_a^b(f+g,\mathcal{F},t).
\end{align*}
Analogously, the second inequality holds.

3.
\begin{align*}
\overset{+}{V}{}_{a}^{b}( f^{+},\mathcal{F},t^{+})_{\mathcal{J}} + \overset{-}{V}{}_{a}^{b}( f^{-},\mathcal{F},t^{-})_{\mathcal{J}}&=\sup_{P\in\mathcal{P}[a,b]}\Bigg\{ \sum_{i=1}^{n}\big[\mathcal{N}^{+}(f^{+}(x_{i})- f^{+}(x_{i-1}),t^{+}) \\
& \quad\quad\quad\quad\quad+\mathcal{N}^{-}( f^{-}(x_{i})- f^{-}(x_{i-1}),t^{-})\big] \Bigg\}\\
&\ge \sup_{P\in\mathcal{P}[a,b]}\left\{ \sum_{i=1}^{n}\mathcal{N}_{\mathcal{J}}(f(t_i)-f(t_{i-1}),t)\right\}=V_{a}^{b}( f,\mathcal{F},t)_{\mathcal{J}} 
\end{align*}
\end{proof}

\begin{proposition}
$BV([a,b],(\mathcal{F},\psi_{\text{stand}}))$ is a vector space.
\end{proposition}
\begin{proof}
Let $f,g\in BV([a,b],(\mathcal{F},\psi_{\text{stand}}))$ and $\alpha\in\mathbb{C}$. Then, there exist $r,s\ge 0$ such that
$$
V_a^b(f,\mathcal{F},t)_{\mathcal{J}}\le r\quad\text{and}\quad V_a^b(g,\mathcal{F},t)_{\mathcal{J}}\le s.
$$
Defining $r+s=\lambda\ge 0$ and $|\alpha | r=\gamma\ge 0$ and applying Theorem \ref{desigualdades}, we obtain
$$
V_a^b(f+g,\mathcal{F},t)_{\mathcal{J}}\le \lambda\quad\text{and}\quad V_a^b(\alpha f,\mathcal{F},t)_{\mathcal{J}}\le \gamma.
$$
Therefore, $BV([a,b],(\mathcal{F},\psi_{\text{stand}}))$ is closed under the sum of functions and the product by scalars, which confirms that it is a vector space.
\end{proof}

\begin{theorem}\label{dota norma}
Let $(\mathcal{F},\psi_{\text{stand}},\mathcal{J})$ be a standardized two-Krein space and $t\in \mathcal{F}$. The application \\
$\mathcal{N}_{\mathcal{BV}_{\mathcal{F}_{t}}}\colon BV([a,b],(\mathcal{F},\psi_{\text{stand}}))\times BV([a,b],(\mathcal{F},\psi_{\text{stand}}))\to [0,\infty )$ defined by
 $$
\mathcal{N}_{\mathcal{BV}_{\mathcal{F}_{t}}}( f ,g)=\mathcal{N}_{\mathcal{J}}(f(a),t)V_{a}^{b}(g,\mathcal{F},t)_{\mathcal{J}}+\mathcal{N}_{\mathcal{J}}(g(a),t)V_{a}^{b}(f,\mathcal{F},t)_{\mathcal{J}},\quad f,g\in BV([a,b],(\mathcal{F},\psi_{\text{stand}})),
 $$
is a two-norm for $BV([a,b],(\mathcal{F},\psi_{\text{stand}}))$.
\end{theorem}
\begin{proof}
Let $f,g\in BV([a,b],(\mathcal{F},\psi_{\text{stand}}))$ and $\alpha\in \mathbb{C}$.

$(2N_{1})$
\begin{align*}
 \mathcal{N}_{\mathcal{BV}_{\mathcal{F}_{t}}}( \alpha f ,g)&=\mathcal{N}_{\mathcal{J}}(\alpha f(a),t)V_{a}^{b}(g,\mathcal{F},t)_{\mathcal{J}}+\mathcal{N}_{\mathcal{J}}(g(a),t)V_{a}^{b}(\alpha f,\mathcal{F},t)_{\mathcal{J}}\\
 &=\left|\alpha\right| \mathcal{N}_{\mathcal{J}}(f(a),t)V_{a}^{b}(g,\mathcal{F},t)_{\mathcal{J}}+\left|\alpha\right|\mathcal{N}_{\mathcal{J}}(g(a),t)V_{a}^{b}(f,\mathcal{F},t)_{\mathcal{J}}\\
 &= \left| \alpha\right|\mathcal{N}_{\mathcal{BV}_{\mathcal{F}_{t}}}( f ,g).
\end{align*}

$(2N_{2})$
\begin{align*}
\mathcal{N}_{\mathcal{BV}_{\mathcal{F}_{t}}}(f+g,h)&=\mathcal{N}_{\mathcal{J}}((f+g)(a),t)V_{a}^{b}(h,\mathcal{F},t)_{\mathcal{J}}+\mathcal{N}_{\mathcal{J}}(h(a),t)V_{a}^{b}(f+g,\mathcal{F},t)_{\mathcal{J}}\\
&\leq \mathcal{N}_{\mathcal{J}}(f(a),t)V_{a}^{b}(h,\mathcal{F},t)_{\mathcal{J}}+\mathcal{N}_{\mathcal{J}}(h(a),t)V_{a}^{b}(f,\mathcal{F},t)_{\mathcal{J}}\\
&\ \ \ +\mathcal{N}_{\mathcal{J}}(g(a),t)V_{a}^{b}(h,\mathcal{F},t)_{\mathcal{J}}+\mathcal{N}_{\mathcal{J}}(h(a),t)V_{a}^{b}(g,\mathcal{F},t)_{\mathcal{J}}\\
&=\mathcal{N}_{\mathcal{BV}_{\mathcal{F}_{t}}}(f,h)+\mathcal{N}_{\mathcal{BV}_{\mathcal{F}_{t}}}(g,h)
\end{align*}

$(2N_{3})$
\begin{align*}
\mathcal{N}_{\mathcal{BV}_{\mathcal{F}_{t}}}( f ,g)&=\mathcal{N}_{\mathcal{J}}(f(a),t)V_{a}^{b}(g,\mathcal{F},t)_{\mathcal{J}}+\mathcal{N}_{\mathcal{J}}(g(a),t)V_{a}^{b}(f,\mathcal{F},t)_{\mathcal{J}}\\
&=\mathcal{N}_{\mathcal{J}}(g(a),t)V_{a}^{b}(f,\mathcal{F},t)_{\mathcal{J}}+\mathcal{N}_{\mathcal{J}}(f(a),t)V_{a}^{b}(g,\mathcal{F},t)_{\mathcal{J}}=\mathcal{N}_{\mathcal{BV}_{\mathcal{F}_{t}}}(g,f).
\end{align*}

\end{proof}

\section{Conclusion}
Every classical Krein space generates a standardized two-Krein space (Proposition \ref{k2k}). In a classical Krein space, the set of positive and negative vectors preserves orthogonality and completeness in the standardized two-Krein space (Propositions \ref{presesrvaort}, \ref{completez}). The two-norms associated with different fundamental decompositions in standardized two-Krein spaces are equivalent (Theorem \ref{2equiv}). The $(2,h)$-variation in two-Hilbert spaces is a particular case of the $t$-variation in standardized two-Krein spaces (Remark \ref{2h2k}). The vanishing of the $t$-variation implies that the two-norm of the function's images is constant with respect to $t$ (Proposition \ref{f cero constante}). The notion of strongly bounded $t$-variation functions in a standardized two-Krein space is independent of the fundamental decompositions (Theorem \ref{varequiv}). Any set of strongly bounded $t$-variation functions in a standardized two-Krein space can be endowed with a two-norm (Theorem \ref{dota norma}).

Given the notable interest generated by classical Krein spaces, it is natural to expect that some previous research, such as that developed in \cite{FERRER-ARROYO-NARANJO, Ferrer-Dominguez-Arroyo}, could be extended to spaces with an indefinite two-metric, which were introduced in this investigation.

%\vskip 6mm
%\noindent{\bf Acknowledgments}

%\noindent   The author is grateful to the reviewers for useful suggestions which improved the contents of this paper.
%The second author was supported by the  National Natural Science Foundation  under Grant No.
%1244145e2.

\end{document}